\documentclass[12pt,reqno]{amsart}
\usepackage[final]{graphicx}
\usepackage{amsfonts}
\usepackage{amsmath}
\usepackage{amssymb}
\usepackage{amsthm}

\newtheorem{theorem}{Theorem}[section]
\newtheorem{lemma}[theorem]{Lemma}
\newtheorem{corollary}[theorem]{Corollary}

\newcommand{\upchi}{\raise1pt\hbox{$\chi$}}
\newcommand{\R}{{\mathord{\mathbb R}}}
\newcommand{\Ha}{{\mathord{\mathbb H}}}
\newcommand{\Sp}{{\mathord{\mathbb S}}}

\begin{document}

\title{Hardy inequalities for fractional integrals on general domains}
\author[Loss]{Michael Loss $^1$}
\author[Sloane]{Craig Sloane $^1$}

\address{ Georgia Institute of Technology, School of Mathematics,
Atlanta GA 30332-0160}
\address{\tt loss@math.gatech.edu , csloane@math.gatech.edu}

\date{\today}
\maketitle

\footnotetext
[1] {Work partially supported by NSF Grant DMS 0901304.\\
\copyright 2009  by the authors. This paper may be reproduced, in its
entirety, for non-commercial purposes.}

\vspace{.3truein}
\centerline{\bf Abstract}

\medskip
{\sl We prove a sharp Hardy inequality for fractional integrals for functions that are
supported in a general domain.
The constant is the same as the one for the half-space and hence our result settles a recent
conjecture of Bogdan and Dyda.}

\section{Introduction}
In this note we prove a conjecture by Bogdan and Dyda \cite{BD} concerning Hardy inequalities
for fractional integrals. It was shown in \cite{BD} that for any function $f $ supported in the half-space
$\Ha^n = \{x \in \R^n \ : \ x=(x_1, \dots, x_n), x_n>0\}$ 
\begin{equation}\label{bogdandyda}
\frac{1}{2} \int_{\Ha^n \times \Ha^n} \frac{|f(x) - f(y)|^2}{|x-y|^{n+\alpha}} dx dy
\ge \kappa_{n,\alpha} \int_{\Ha^n} \frac{|f(x)|^2}{ x_n^\alpha} dx \ .
\end{equation}
Here $0 < \alpha < 2$ and 
\begin{equation}
\kappa_{n,\alpha} =     \pi^{\frac{n-1}{2}} \frac{\Gamma(\frac{1+\alpha}{2})}{\Gamma(\frac{n+\alpha}{2})} \frac{1}{\alpha}
\left[\frac{2^{1-\alpha}}{\sqrt \pi} \Gamma(\frac{2-\alpha}{2})\Gamma(\frac{1+\alpha}{2}) - 1 \right] 
\end{equation}
is the sharp constant. Note that $\kappa_{n,1}=0$ and $\kappa_{n,\alpha} >0$ otherwise.

It was conjectured in \cite{BD} that for $1 < \alpha < 2$  this inequality continues to hold with the same constant 
for any convex set $\Omega$, i.e., for functions $f$ supported in $\Omega$
\begin{equation} \label{convex}
\frac{1}{2} \int_{\Omega \times \Omega} \frac{|f(x) - f(y)|^2}{|x-y|^{n+\alpha}} dx dy
\ge \kappa_{n,\alpha} \int_\Omega \frac{|f(x)|^2}{d_\Omega(x)^\alpha} dx \ ,
\end{equation}
where $d_\Omega(x)$ denotes the distance from the point $x \in \Omega$ to the boundary of $\Omega$. 
This is a precise analogue  of the Hardy inequality due to Davies \cite{Davies2}.
For $0 < \alpha < 1$ the inequality cannot hold for compact sets. A counterexample is given in  \cite{D}.

Sharp Hardy inequalities analogous to (\ref{convex}) but  for  the  $L^p$-norms
of gradients of functions are well known. The first result is due to Davies \cite{Davies2} for the case $p=2$.
The case for arbitrary $p$ is derived in \cite{MS} and \cite{MMP}. For a review the reader may consult \cite{Davies3}.
Let us add that these results have been considerably
generalized in \cite{BFT}.

Hardy inequalities for fractional integrals are of a more recent provenience, in particular the higher dimensional
versions were investigated by Dyda (see \cite{D}) in great generality following previous work in \cite{HKP} and \cite{CS}.  
While Hardy inequalities for fractional integrals are of interest in their own right, they deliver also spectral information
on the generators of censored stable processes. The generator of a censored stable process is defined
by the closure of the quadratic form on the left side of (\ref{convex}). Loosely speaking it is a stable
process with the jumps between $\Omega$ and its complement suppressed. 
The reference \cite{Burdzy} contains the construction of censored stable processes and a wealth of information about these.  
For the connection between Hardy inequalities and censored stable processes the reader may consult \cite{CS}.

Since we prove a stronger result than (\ref{convex}) we need a few concepts before we can state the result. 
Let $\Omega$ be any domain in $\R^n$ with non-empty boundary. The following notion is taken from Davies \cite{Davies}. Fix a direction $w \in \Sp^{n-1}$ and define
\begin{equation}
d_{w,\Omega}(x) = \min\{ |t|\ : \ x+tw \notin \Omega\} \ .
\end{equation}
Further, define the function
\begin{equation} \label{delta}
\delta_{w,\Omega}(x) = \sup\{|t| : x+tw \in \Omega\} \ ,
\end{equation}
i.e.,  $\delta_{w,\Omega}(x)$ is the point in the intersection 
of the line $x+tw$ and $\Omega$ that is farthest away from $x$
and set
\begin{equation}
\frac{1}{M_\alpha (x)^\alpha} := \frac{ \int_{\Sp^{n-1}} dw \left[  \frac{1}{d_{w, \Omega}(x)}
+ \frac{1}{\delta_{w, \Omega}(x)} \right] ^\alpha}{ \int_{\Sp^{n-1}} dw  |w_n|^\alpha} \ .
\end{equation}
The integral in the denominator can be easily computed to be
\begin{equation}\label{integral}
\int_{\Sp^{n-1}} dw  {|w_n|^\alpha} = 2 \pi^{\frac{n-1}{2}} \frac{\Gamma(\frac{1+\alpha}{2})}{\Gamma(\frac{n+\alpha}{2})}
\end{equation}
These definitions are analogous to the one in \cite{Davies} where all estimates are expressed in terms of 
$$
\frac{1}{m_2(x)^2} =  \frac{\int_{\Sp^{n-1}} dw \frac{1}{d_{w, \Omega}(x)^2}}{ |\Sp^{n-1}|/n} \ .
$$
In case the domain $\Omega$ is convex, the quantity $M_\alpha(x)$ can be bounded in terms of $d_\Omega(x)$ and
$D_\Omega(x)$, the {\it `width of $\Omega$ with respect to $x$'}. For convex domains with smooth boundary,
this quantity is given by the width of the smallest slab that contains $\Omega$ and consists of two parallel hyper-planes one of which is tangent
to $\partial \Omega$ at the point closest to $x$. For general convex sets we define it as follows. 
 Fix $x \in \Omega$ arbitrary and pick a  point $z$ on the boundary of $\Omega$ that is closest to $x$, so that
$d_\Omega(x)  = |x-z|$. In general, there may be more than one such point. 
Denote by $P_z$ the set of supporting  hyper-planes of $\Omega$ that pass through the  point $z$ and set
$$
\mathcal P_x = \cup_{z \in \partial \Omega, |z-x| = d_\Omega (x)} P_z \ .
$$
For $P \in \mathcal P_x$, we denote by $S(P)$ the smallest slab that contains $\Omega$ and is bounded by $P$ 
on one side and a hyper-plane parallel to it on the other.
Such a slab might be a half space if $\Omega$ is unbounded.
The width $D_{S(P)}$ of the slab $S(P)$ is, naturally, the distance between the two bounding hyper-planes.  
We set $D_{S(P)} = \infty$ if $S(P)$ is a half space.
Now we define  
\begin{equation} \label{width}
D_\Omega(x) = \inf_{P \in \mathcal P_x}  D_{S(P)} \ .
\end{equation}

The inequality
\begin{equation} \label{useful}
\frac{1}{M_\alpha (x)^\alpha}  \ge \left[ \frac{1}{d_\Omega(x)} + \frac{1}{D_\Omega (x) -d_\Omega(x)}\right]^\alpha \ ,
\end{equation}
 follows from
\begin{equation} \label{easytosee}
\int_{\Sp^{n-1}} dw \left[ \frac{1}{d_{w,\Omega}(x)} + \frac{1}{\delta_{w,\Omega}(x)}\right]^\alpha
\ge  \int_{\Sp^{n-1}}dw  |w_n|^\alpha  \left[ \frac{1}{d_\Omega(x)} + \frac{1}{D_\Omega (x) -d_\Omega(x)}\right]^\alpha \ .
\end{equation}
Indeed, for given $P$ pick coordinates such that the standard vector $e_n$ is normal to the plane $P$.
Clearly  $ d_{w,\Omega}(x) \le d_{w,S(P)}(x)$ and $ \delta_{w,\Omega}(x) \le \delta_{w,S(P)}(x)$.
Further, note that $d_{w,S(P)}(x)+\delta_{w,S(P)}(x)$ is the length of the segment given by intersecting the slab $S(P)$
with the line $x+tw$. 
Projecting this segment onto the line normal to the slab yields
$$
d_{w,S(P)}(x) |w_n| = d_\Omega(x) \ , \  \delta_{w,S(P)}(x)|w_n| = D_{S(P)} - d_\Omega(x)  \ .
$$
Note that there may exist directions $w$ where the length of this segment is not finite in which case we set
$D_{S(P)}=\infty$.
Thus,
$$
\left[ \frac{1}{d_{w,\Omega}(x)} + \frac{1}{\delta_{w,\Omega}(x)}\right]^\alpha
\ge  |w_n|^\alpha  \left[ \frac{1}{d_\Omega(x)} + \frac{1}{D_{S(P)} -d_\Omega(x)}\right]^\alpha 
$$
holds for all $P \in \mathcal P_x$. 
Taking the supremum over $\mathcal P_x$ and integrating with respect to 
$w$ over the unit sphere yields  (\ref{easytosee}).
With these preparations we can state our main theorem. 
\begin{theorem} \label{main}
Let $\Omega$  be a domain with non-empty boundary and $1 < \alpha < 2$. For any $f \in C^\infty_c(\Omega)$ 
\begin{equation} \label{fundi}
\frac{1}{2} \int_{\Omega \times \Omega} \frac{|f(x) - f(y)|^2}{|x-y|^{n+\alpha}} dx dy
\ge \kappa_{n,\alpha} \int_\Omega \frac{|f(x)|^2}{M_\alpha(x)^\alpha} dx \ .
\end{equation}
In particular, if $\Omega$ is a convex region then for any $f \in C^\infty_c(\Omega)$
\begin{equation} \label{fundiconvex}
\frac{1}{2} \int_{\Omega \times \Omega} \frac{|f(x) - f(y)|^2}{|x-y|^{n+\alpha}} dx dy
\ge  \kappa_{n,\alpha} \int_\Omega {|f(x)|^2}\left[ \frac{1}{d_\Omega(x)} + \frac{1}{D_\Omega (x) -d_\Omega(x)}\right]^\alpha dx
\end{equation}
where $d_\Omega(x)$ is the distance of $x \in \Omega$ to the boundary of $\Omega$ and $D_\Omega(x)$ is defined in (\ref{width}).  The constant $\kappa_{n,\alpha}$ is best possible.
\end{theorem}

It was pointed out to us by Rupert Frank and Robert Seiringer that this Theorem \ref{main} can be generalized,
albeit in a weaker form, by replacing the powers
$2$ by $p > 1$. More precisely we have,

\begin{theorem}\label{frasei} Let $1 < p < \infty $ and $1< \alpha < p$. Then for any domain $\Omega \subset \R^n$ and any $f \in C^\infty_c(\Omega)$
\begin{equation}
\int_{\Omega \times \Omega} \frac{|f(x) - f(y)|^p}{|x-y|^{n+\alpha}} dx dy
\ge \mathcal{D}_{n,p,\alpha} \int_\Omega \frac{|f(x)|^p}{m_{\alpha}(x)^{\alpha}} dx
\end{equation}
where
\begin{equation}
\frac{1}{m_{\alpha}(x)^{\alpha}} :=  \frac{ \int_{\Sp^{n-1}} dw  \frac{1}{d_{w, \Omega}(x)^{\alpha}}}{ \int_{\Sp^{n-1}} dw  |w_n|^{\alpha}} \ .
\end{equation}
and
\begin{equation}
\mathcal{D}_{n,p,\alpha} = 2 \pi^{\frac{n-1}{2}} \frac{\Gamma(\frac{1+\alpha}{2})}{\Gamma(\frac{n+\alpha}{2})}  
 \int_0^1 \frac{|1-r^{\frac{\alpha-1}{p}}|^p}
 {(1-r)^{1+\alpha}} dr 
 \end{equation}
is the sharp constant. In particular, for $\Omega$ convex 
\begin{equation}
\int_{\Omega \times \Omega} \frac{|f(x) - f(y)|^p}{|x-y|^{n+\alpha}} dx dy
\ge \mathcal{D}_{n,p,\alpha} \int_\Omega \frac{|f(x)|^p}{d_\Omega(x)^{\alpha}} dx \ .
\end{equation}
\end{theorem}
The constant $\mathcal{D}_{n,p,s}$ has been computed before in \cite{FS} as the sharp constant for the Hardy inequality
for the half-space. For $0 < p \le 1$ the inequality continuous to hold (see \cite{D}), however, the sharp constant is not known.

In the next section we establish the analogous one dimensional inequalities and then show how an averaging
argument leads to the general result. At the end of Section \ref{onedim} we indicate how to obtain the result
for general values of $p$.
We are grateful to Rupert Frank and Robert Seiringer to allow us to include arguments in our work. 
We present them at the end of our paper.

\medskip
\noindent
{\bf Acknowledgment:}  M.L. would like to
thank the Erwin Schr\"odinger Institut for its kind hospitality and useful discussions with Thomas Hoffmann-Ostenhof.
The authors thank an anonymous referee for suggesting various improvements of the manuscript.

\section{The one dimensional problem} \label{onedim}

The proof of Theorem \ref{main} will rely heavily on the following one dimensional inequality. 
\begin{theorem}  \label{interval} 
Let $f \in C^\infty_c((a,b))$. Then for all $1 < \alpha < 2$ we have
 \begin{equation} 
 \frac{1}{2}\int_{(a,b) \times (a,b)} \frac{|f(x) -f(y)|^2}{|x-y|^{1+\alpha}} dx dy \ge  \kappa_{1,\alpha}  \int_a^b  |f(x)|^2 \left(\frac{1}{x-a}
 +\frac{1}{b-x} \right)^\alpha dx \ .
 \end{equation}
 \end{theorem}
The idea of proving Theorem \ref{interval} is to reduce the problem on the interval to a problem on the
half-line via a fractional linear mapping. The reader may consult \cite{calo} for further examples where
inversion symmetry is used to obtain sharp functional inequalities.

\begin{lemma}[Invariance under fractional linear transformations]  \label{conformal}
Let $f$ be any function in $C^\infty_c(\R \setminus \{0\})$.
Consider the inversion $x \to 1/x$ and set
$$
g(x) = I(f)(x) := |x|^{\alpha -1} f(\frac{1}{x}) \ .
$$
Then $g \in C^\infty_c(\R)$ and 
\begin{equation} \label{invariance}
\int_{\R \times \R} \frac{|g(x) -g(y)|^2}{|x-y|^{1 + \alpha}} dx dy  = \int_{\R \times \R} \frac{|f(x) -f(y)|^2}{|x-y|^{1 + \alpha}} dx dy \ .
\end{equation}
\end{lemma}
\begin{proof}
For fixed $ \varepsilon$ consider  the regions
$$
R_1 := \{(x,y) \in \R^2: |\frac{x}{y}| > 1+\varepsilon \} \ , 
$$
and likewise,
$$
R_2 := \{(x,y) \in \R^2:  |\frac{y}{x}| > 1+\varepsilon  \} \ .
$$
Now by changing variables $x \to 1/x$ and $y \to 1/y$ we find that
$$
\int_{R_1 \cup R_2}  \frac{|f(x)-f(y)|^2}{|x-y|^{1 +\alpha}} dx dy
= \int_{R_1 \cup R_2}  \frac{|f(1/x)-f(1/y)|^2}{|x-y|^{1 +\alpha}} |x|^{\alpha-1} |y|^{\alpha-1}dx dy
$$
$$
= \int_{R_1 \cup R_2}  \frac{| g(x)-g(y)|^2}{|x-y|^{1 +\alpha}} dx dy
$$
$$
+\int_{R_1 \cup R_2}  \frac{|f(1/x)|^2(|x|^{\alpha-1} |y|^{\alpha-1} -|x|^{2(\alpha-1)}) +|f(1/y)|^2(|x|^{\alpha-1} |y|^{\alpha-1} -|y|^{2(\alpha-1)})}{|x-y|^{1 +\alpha}} dx dy
$$
which, by symmetry under exchange of $x$ and $y$,
$$
=  \int_{R_1 \cup R_2}  \frac{| g(x)-g(y)|^2}{|x-y|^{1 +\alpha}} dx dy
+2\int_{R_1 \cup R_2}  \frac{|f(1/x)|^2(|x|^{\alpha-1} |y|^{\alpha-1} -|x|^{2(\alpha-1)}) }{|x-y|^{1 +\alpha}} dx dy \ .
$$
We can write the second term as
$$
\int_\R |f(1/x)|^2  |x|^{\alpha-2} \int_{\{  |s| > 1+ \varepsilon\} \cup \{ \frac{1}{|s|} > 1+ \varepsilon\}} \frac{|s|^{\alpha -1} -1}{|1-s|^{1+\alpha}} ds \ .
$$
The integral
$$
\int_{\{  |s| > 1+ \varepsilon\} \cup \{  \frac{1}{|s|} > 1+ \varepsilon\}} \frac{|s|^{\alpha -1} -1}{|1-s|^{1+\alpha}} ds
$$
$$
= \int_{\{ |s| > 1+ \varepsilon\}} \frac{|s|^{\alpha -1} -1}{|1-s|^{1+\alpha}} ds
+\int_{ \{ \frac{1}{|s|}> 1+ \varepsilon\}} \frac{|s|^{\alpha -1} -1}{|1-s|^{1+\alpha}} ds 
$$
and by changing the variable $s \to 1/s$ in the last integral we find that this sum vanishes. Letting $\varepsilon \to 0$ yields
(\ref{invariance}).

\end{proof}
\begin{proof}[Proof of Theorem \ref{interval}]
By translation and scaling it suffices to prove the result for the interval $(0,1)$. 
Let $f \in C^\infty_c((0,1))$. We have to show that
\begin{equation} \label{fundamental}
\frac{1}{2}\int_{(0,1) \times(0,1)} \frac{|f(x) - f(y)|^2}{|x-y|^{1+\alpha}} dx dy \ge \kappa_{1,\alpha}
\int_0^1 |f(x)|^2 \left(\frac{1}{x} + \frac{1}{1-x}\right)^\alpha dx \ .
\end{equation}
Set
$$
g(x) = |x+1|^{\alpha -1} f(\frac{1}{1+x}) \ .
$$
Clearly, $g \in C^\infty_c((0,\infty))$. Note that
$$
g(x) = I(f)(x+1)	
$$
and hence we may use Lemma \ref{conformal} and find that
\begin{eqnarray}
& &\frac{1}{2}\int_0^1\int_0^1 \frac{|f(x) - f(y)|^2}{|x-y|^{1+\alpha} } dx dy+  \int_0^1 dx |f(x)|^2 \int_{\R \setminus (0,1)} \frac{1}{|x-y|^{1+\alpha}} dy
\nonumber \\
&=& \frac{1}{2}\int_{\R \times \R} \frac{|f(x) - f(y)|^2}{|x-y|^{1+\alpha} } dx dy =  \frac{1}{2}\int_{\R \times \R} \frac{|g(x) - g(y)|^2}{|x-y|^{1+\alpha} } dx dy 
\nonumber \\
&=&\frac{1}{2} \int_0^\infty  \int_0^\infty \frac{|g(x) - g(y)|^2}{|x-y|^{1+\alpha} } dx dy  + \int_0^\infty dx |g(x)|^2 \int_{-\infty}^0 \frac{1}{|x-y|^{1+\alpha} } dy
\end{eqnarray}
Some of the integrals are easily evaluated and yield
\begin{eqnarray}
& &\frac{1}{2} \int_0^1\int_0^1 \frac{|f(x) - f(y)|^2}{|x-y|^{1+\alpha} } dx dy = \frac{1}{2} \int_0^\infty  \int_0^\infty \frac{|g(x) - g(y)|^2}{|x-y|^{1+\alpha} } dx dy
\nonumber \\
&+& \frac{1}{\alpha} \int_0^\infty  \frac{|g(x)|^2}{x^\alpha} dx - \frac{1}{\alpha}  \int_0^1 |f(x)|^2 \left(x^{-\alpha} + (1-x)^{-\alpha}\right) dx \ .
\end{eqnarray}
Using the sharp Hardy inequality of Bogdan and Dyda \cite{BD} on the half-line yields
\begin{eqnarray}
& &\frac{1}{2} \int_0^1\int_0^1 \frac{|f(x) - f(y)|^2}{|x-y|^{1+\alpha} } dx dy \ge \kappa_{1,\alpha} \int _0^\infty \frac{|g(x)|^2}{x^\alpha} dx \nonumber \\
&+& \frac{1}{\alpha} \int_0^\infty  \frac{|g(x)|^2}{x^\alpha} dx - \frac{1}{\alpha}  \int_0^1 |f(x)|^2 \left(x^{-\alpha} + (1-x)^{-\alpha}\right) dx
\end{eqnarray}
Changing variables, i.e., expressing everything in terms of the function $f$, we arrive at the inequality
\begin{eqnarray}\label{interesting}
& &\frac{1}{2}\int_0^1\int_0^1 \frac{|f(x) - f(y)|^2}{|x-y|^{1+\alpha} } dx dy  \nonumber \\
&\ge&   \kappa_{1,\alpha} \int _0^1 |f(x)|^2 \left( \frac{1}{x(1-x)}\right)^\alpha dx+
\frac{1}{\alpha} \int _0^1 |f(x)|^2 \frac{1-x^\alpha -(1-x)^\alpha}{(x(1-x))^\alpha} dx
\end{eqnarray}
Finally, we note that for $1 < \alpha < 2$ 
$$
1-x^\alpha -(1-x)^\alpha \ge 0
$$
which proves the inequality (\ref{fundamental}). 
\end{proof}
Theorem \ref{interval} generalizes easily to open sets on the real line.
  \begin{corollary} \label{open}
 Let $J \subset \R$ be an open set and $1 < \alpha < 2$. Then for any $f \in C^\infty_c(J)$  
 \begin{equation}
 \frac{1}{2}\int_{J \times J} \frac{|f(x) -f(y)|^2}{|x-y|^{1+\alpha}} dx dy \ge  \kappa_{1,\alpha}  \int_J |f(x)|^2 \left( \frac{1}{d_J(x)} + \frac{1}{\delta_J(x) } \right)^\alpha dx \ ,
 \end{equation}
 where $\delta_J(x)$ is defined in (\ref{delta}).
 \end{corollary}

\begin{proof}
Since any open set $J \subset \R$ is a countable union of disjoint intervals $I_k$ we find, using Theorem \ref{interval}, that
\begin{eqnarray}
& &\frac{1}{2} \int_J\int_J \frac{|f(x) - f(y)|^2}{|x-y|^{1+\alpha} } dx dy \ge \frac{1}{2} \sum_{k=1}^\infty \int_{I_k} \int_{I_k}  \frac{|f(x) - f(y)|^2}{|x-y|^{1+\alpha} } dx dy
\nonumber \\
&\ge& \sum_{k=1}^\infty  \kappa_{1,\alpha} \int_{I_k}  |f(x)|^2 \left(\frac{1}{d_{I_k}(x)} + \frac{1}{\delta_{I_k}(x)}\right)^\alpha dx \nonumber \\
&\ge& \kappa_{1,\alpha} \int_J |f(x)|^2 \left(\frac{1}{d_J(x)} + \frac{1}{\delta_J(x) }\right)^\alpha dx \ .
\end{eqnarray}
\end{proof}

\

\begin{lemma}[Reduction to one dimension] \label{reduction}
Let $\Omega$ be any region in $\R^n$ and assume that $f \in C^\infty_c(\Omega)$. Then
\begin{eqnarray} \label{averagetwo}
& &\int_{\Omega \times \Omega} \frac{|f(x) - f(y)|^p}{|x-y|^{n+\alpha}} dx dy  \nonumber\\
&=&\frac{1}{2} \int_{\Sp^{n-1}} dw \int_{\{x: x \cdot w = 0\}} d{\mathcal L}_w(x) \int_{x+sw \in \Omega}  ds \int_{x+tw \in \Omega}dt
\frac{|f(x+sw) - f(x+tw)|^p}{{|s-t|}^{1+\alpha}} 
\end{eqnarray}
where ${\mathcal L}_w$ denotes the $(n-1)$ dimensional Lebesgue measure on the plane $x \cdot w =0$.
\end{lemma}
\begin{proof} We write the expression
$$
I_\Omega (f):=  \int_{\Omega \times \Omega} \frac{|f(x) - f(y)|^p}{|x-y|^{n+\alpha}} dx dy
$$
in the form
$$
\int_\Omega dx  \int_ {\{ x+z \in \Omega\}}  dz \frac{|f(x) - f(x+z)|^p}{|z|^{n+\alpha}}
$$
and using polar coordinates $z = r w$ we arrive at the expression
\begin{eqnarray}
I_\Omega (f) &=& \int_\Omega dx  \int_{\Sp^{n-1}} dw \int_ {\{ x+rw \in \Omega \ ,\  r >0\}}  dr \frac{|f(x) - f(x+rw)|^p}{r^{1+\alpha}} \ , \nonumber \\
&=& \frac{1}{2} \int_{\Sp^{n-1}} dw \int_\Omega dx   \int_ {\{ x+hw \in \Omega \}}  dh \frac{|f(x) - f(x+hw)|^p}{|h|^{1+\alpha}}  \ .
\end{eqnarray}
Thus, the domain of integration in the innermost integral is the line $x +hw$ intersected with the domain $\Omega$. Splitting the variable $x$ into
a component perpendicular to $w$ and parallel to $w$, i.e., replacing $x$ by $x+sw$, where $x \cdot w = 0$, we arrive at
$$
\frac{1}{2} \int_{\Sp^{n-1}} dw\int_{\{ x : x \cdot w=0\}} d{\mathcal L}_w(x)  \int_{\{x+sw \in \Omega\}} ds  \int_ {\{ x+(s+h)w \in \Omega \}}  dh \frac{|f(x+sw) - f(x+(s+h)w)|^p}{|h|^{1+\alpha}} 
$$
The variable change $t = s+h$ yields (\ref{averagetwo}).
\end{proof}

\begin{proof}[Proof of Theorem \ref{main}]
By Lemma \ref{reduction} and Corollary \ref{open} we find that
\begin{eqnarray}
& &\frac{1}{2}\int_{\Omega \times \Omega} \frac{|f(x) - f(y)|^2}{|x-y|^{n+\alpha}} dx dy \nonumber \\
&=&\frac{1}{4} \int_{\Sp^{n-1}} dw \int_{\{x: x \cdot w = 0\}} d{\mathcal L}_w(x) \int_{x+sw \in \Omega}  ds \int_{x+tw \in \Omega}dt
\frac{|f(x+sw) - f(x+tw)|^2}{{|s-t|}^{1+\alpha}}  \nonumber \\
&\ge&\kappa_{1,\alpha}  \frac{1}{2} \int_{\Sp^{n-1}} dw \int_{\{x: x \cdot w = 0\}} d{\mathcal L}_w(x) \int _{x+sw \in \Omega}  ds |f(x+sw)|^2 \left[\frac{1}{d_w(x+sw)}+ \frac{1}{\delta_w(x+sw)}\right]^\alpha  \nonumber \\
&=& \kappa_{1,\alpha}  \frac{1}{2} \int_{\Sp^{n-1}} dw  \int_\Omega |f(x)|^2\left[\frac{1}{d_{w,\Omega} (x)}+ \frac{1}{\delta_{w, \Omega}(x)} 
\right]^\alpha dx
 = \kappa_{n,\alpha}  \int_\Omega \frac{|f(x)|^2}{M_\alpha (x)^\alpha} dx \ ,
\end{eqnarray}
where we have used (\ref{integral}) in the last equation.  It remains to show that the constant $\kappa_{n, \alpha}$ in the inequality (\ref{fundiconvex}) 
is best possible. Pick a hyper-plane $H$ that is tangent to $\Omega$ at a point $P$. Such hyper-planes exist  since $\Omega$ is convex. It was shown in \cite{BD}
that the constant for the half-space problem, $\kappa_{n,\alpha}$, is best possible by constructing a sequence of trial functions. Transplanting these trial functions to $\Omega$ near the point $P$ one can show that $\kappa_{n,\alpha}$ is also optimal for  (\ref{fundiconvex}). 
The actual proof is a straightforward imitation of the proof of Theorem 5 in \cite{MMP} and we omit the details. 
\end{proof}
We finally come to the proof of Theorem \ref{frasei}. We thank Rupert Frank and Robert Seiringer for allowing us to present their argument.
\begin{theorem}
Let $1 < p < \infty $ and $1 < \alpha < p$. Then for all smooth functions $f$ with $f(0)=0$,
$$
\int_0^1 \int_0^1 \frac{|f(x)-f(y)|^p}{|x-y|^{1+\alpha}} \,dx\,dy 
\geq \mathcal D_{1,p,\alpha} \int_0^1 \frac{|f(x)|^p}{x^{\alpha}}\,dx \ .
$$
\end{theorem}

\begin{proof}
Let $\omega(x)=x^{(\alpha-1)/p}$. Then by \cite[Lemma 2.4]{FS}
\begin{equation*}
 2 \int_0^\infty 
\left(\omega(x) -\omega(y)\right)
\left|\omega(x) - \omega(y) \right|^{p-2} \frac{dy}{|x-y|^{1+\alpha}}
= \frac{\mathcal D_{1,p,\alpha}}{x^{\alpha}} \ \omega(x)^{p-1}
\end{equation*}
where the integral is understood in principal value sense. Since
$$
\int_1^\infty 
\left(\omega(x) -\omega(y)\right)
\left|\omega(x) - \omega(y) \right|^{p-2} \frac{dy}{|x-y|^{1+\alpha}} \leq 0
\quad \text{for}\ x\in [0,1] \,,
$$
we conclude that
$$
V(x) := \frac{2}{\omega(x)^{p-1}} \int_0^1
\left(\omega(x) -\omega(y)\right)
\left|\omega(x) - \omega(y) \right|^{p-2} \frac{dy}{|x-y|^{1+\alpha}}
\geq \frac{\mathcal D_{1,p,\alpha}}{x^{\alpha}}
\quad \text{for}\ x\in [0,1] \,.
$$
Now \cite[Prop. 2.2]{FrSe1} implies that
$$
\int_0^1 \int_0^1 \frac{|f(x)-f(y)|^p}{|x-y|^{1+\alpha}} \,dx\,dy 
\geq \int_0^1 V(x) |f(x)|^p \,dx \,,
$$
which proves the claim.
\end{proof}
An easy consequence is 
\begin{theorem}  \label{pinterval} 
Let $f \in C^\infty_c((a,b))$. Then for all $1 < p < \infty $ and $1 < \alpha  < p$ we have
 \begin{equation} 
\frac{1}{2}\int_{(a,b) \times (a,b)} \frac{|f(x) -f(y)|^p}{|x-y|^{1+\alpha}} dx dy \ge  \mathcal{D}_{1,p,\alpha}  \int_a^b   \frac{|f(x)|^p}{\min\{(x-a),(b-x)\}^{\alpha}} dx  \ .
 \end{equation}
 
 \end{theorem}
Exactly the same proof as the one of Corollary \ref{open} yields
\begin{corollary}
Let $1 < p< \infty $ and $1 < \alpha < p  $. Let $J\subset\R$ be open and $f$ a function on $J$ with $f \in C^\infty_c(J)$ , then
$$
\int_J \int_J \frac{|f(x)-f(y)|^p}{|x-y|^{1+\alpha}} \,dx\,dy 
\geq \mathcal D_{1,p,\alpha} \int_J \frac{|f(x)|^p}{d_J(x)^{\alpha}}\,dx \,.
$$
\end{corollary}

\begin{proof}[Proof of Theorem \ref{frasei}]
The proof is a repetition of the arguments in the proof of Theorem \ref{main}.
\end{proof}


\begin{thebibliography}{10}

\bibitem{BFT} Barbatis, G.,  Filippas, S.,  Tertikas, A. A,  unified approach to improved $L^p$ Hardy inequalities with best constants. Trans. Amer. Math. Soc. 356 (2004), no. 6, 2169--2196
\bibitem{Burdzy} Bogdan, K., Burdzy, K. and Chen, Z-Q., Censored stable processes, Probab. Theory Relat. Fields, {\bf 127}, 89-152, 2003.

\bibitem{BD} Bogdan, K. and Dyda, B.,  The best constant in a fractional Hardy inequality, 
arXiv:0807.1825. 

\bibitem{calo} Carlen, E. A.  and Loss, M.,  On the minimization of symmetric functionals, Reviews in
Mathematical Physics, Special Issue, (1994) 1011--1032. 


\bibitem{D}  Dyda, B., A fractional order Hardy inequality, Illinois J. of Math., {\bf 48}, 575-588, 2004.

\bibitem{CS}  Chen, Z-Q. and Song, R., Hardy inequality for censored stable processes, Tohoku Math. J. {\bf 55}, 439-450, 2003.

\bibitem{Davies} Davies, E.B., Heat kernels and spectral theory, Cambridge Tracts in Mathematics, Cambridge University Press, 1989.

\bibitem{Davies2}  Davies, E.B., Some norm bounds and quadratic form inequalities for Schr\"odinger operators, II, J. Oper. Theory {\bf 12}, 177-196, 1984.

\bibitem{Davies3} Davies, E. B. A review of Hardy inequalities. The Maz'ya anniversary collection, Vol. 2 (Rostock, 1998), 55--67, Oper. Theory Adv. Appl., 110, BirkhŠuser, Basel, 1999. 

\bibitem{FrSe1} R. L. Frank, R. Seiringer,  Non-linear ground state representations and sharp Hardy inequalities . J. Funct. Anal. 255 (2008), 3407 - 3430.


\bibitem{FS}  Frank, R. and Seiringer, R., Sharp fractional Hardy inequalities in half-spaces, arXiv:0906.1561.

\bibitem{HKP}  H. P. Heinig, A. Kufner, and L.-E. Persson, On some fractional order Hardy in-
equalities, J. Inequal. Appl. 1 (1997), 25Ð46. 

\bibitem{MMP} Marcus M., Mizel V.J. and Pinchover Y. , On the best constant for HardyÕs inequality in
$\R^ n$ . Trans. Amer. Math. Soc. 350 (1998), 3237-3255. 


\bibitem{MS}  Matskewich T. and Sobolevskii P.E. , The best possible constant in generalized HardyÕs inequality for convex domain in $\R^n$. Nonlinear Anal., Theory, Methods and Appl., 28 (1997), 1601-1610. 



\end{thebibliography}
\end{document}